\documentclass{amsart}

\usepackage{amsmath,amsthm,amssymb}

\theoremstyle{plain}

\newtheorem{definition}{Definition}
\newtheorem{thm}[definition]{Theorem}

\newtheorem{lem}[definition]{Lemma}
\newtheorem{cor}[definition]{Corollary}

\begin{document}
\title[group determinant and group algebra]{Factorization of group determinant in some group algebras}
\author[N. Yamaguchi]{Naoya Yamaguchi}
\date{\today}

\maketitle

\begin{abstract}
We give an analog of Frobenius' theorem about the factorization of the group determinant on the group algebra of finite abelian groups 
and we extend it into dihedral groups and generalized quaternion groups. 
Furthermore, we describe the group determinant of dihedral groups and generalized quaternion groups as a circulant determinant of homogeneous polynomials. 

This analog on the group algebra is stronger than Frobenius's theorem 
and as a corollary, we obtain a simple expression formula for inverse elements in the group algebra. 
Furthermore, the commutators of irreducible factors of the factorization of the group determinant on the group algebra corresponding to degree one representations have interesting algebraic properties. 
From this result, we know that degree one representations form natural pairing.

At the current stage, the extension of Frobeinus' theorem is not represent as a determinant. 
We expect to find a determinant expression similar to Frobenius' theorem. 
\end{abstract}

\section{\bf{INTRODUCTION}}
In this paper, we give an analog of the factorization of the group determinant on the group algebra, 
where the group $G$ is a finite abelian group, a dihedral group or a generalized quaternion group. 
The group determinant $\Theta(G)$ is the determinant of the matrix whose elements are independent variables $x_{g}$ corresponding to $g \in G$. 
Frobenius gave the following theorem about the factorization of the group determinant. 

\begin{thm}[Frobenius]\label{thm:1.1.1}
Let $G$ be a finite group and $\widehat{G}$ be a complete set of irreducible representations of G. 
Then, we have 
$$
\Theta(G) = \prod_{\varphi \in \widehat{G}} \det{\left( \sum_{g \in G} \varphi (g) x_{g} \right)^{\deg{\varphi}}}. 
$$
\end{thm}
Our main results of this paper are non-trivial extensions of this theorem for some group on the corresponding group algebras.

\subsection{Results for finite abelian groups}
Our main result for finite abelian group is the following theorem. 
\begin{thm}\label{thm:1.1.2}
Let $G$ be a finite abelian group and $e$ be the unit element of $G$. 
Then, we have 
$$
\Theta(G) e = \prod_{\chi \in \widehat{G}} \sum_{g \in G} \chi(g) x_{g} g. 
$$
\end{thm} 
Let 
$$
\mathbb{C} G = \left\{ \sum_{g \in G} x_{g} g \: \vert \: x_{g} \in \mathbb{C}, g \in G \right\}
$$ 
be the group algebra of G. 
Note that the equality in Theorem~$\ref{thm:1.1.2}$ is the equality on the group algebra. 
Theorem~$\ref{thm:1.1.2}$ is stronger than Theorem~$\ref{thm:1.1.1}$. 
In fact, let $F: \mathbb{C} G \rightarrow \mathbb{C}$ be a $\mathbb{C}$-linear ring homomorphism such that $F(g) = 1$ for all $g \in G$. 
Theorem $\ref{thm:1.1.1}$ for finite abelian groups follows by applying $F$ to Theorem~$\ref{thm:1.1.2}$. 
Moreover, we obtain the following formula for inverse elements in the group algebra. 

\begin{cor}\label{cor:1.1.3}
Let $G$ be a finite abelian group and $\chi_{1}$ be the trivial representation of $G$. 
If $\Theta(G) \neq 0$, then we have 
$$
\left( \sum_{g \in G} x_{g} g \right)^{-1} = \frac{1}{\Theta(G)} \prod_{\chi \in \widehat{G} \setminus \{ \chi_{1} \} } \left( \sum_{g \in G} \chi (g) x_{g} g \right). 
$$
\end{cor}

\subsection{Main results for dihedral groups and generalized quaternion groups}
We give an analog of Theorem~$\ref{thm:1.1.1}$ on the group algebra for dihedral groups and generalized quaternion groups. 
Furthermore, 
we describe the group determinant of these groups as a circulant determinant of homogeneous polynomials. 
Since the expression for the general case is troublesome, 
we illustrate the result on a dihedral group $D_{3} = \{ e, a, a^{2}, b, ab, a^{2} b \}$. 
Let $\langle a \rangle$ be the subgroup generated by $\{ a \}$ and $A_{h}$ be a homogeneous polynomial defined by 
$$
A_{h} = \sum_{g \in \langle a \rangle} \chi_{2}(g) x_{g} x_{h g}
$$
for $h \in G$. 
We describe $\Theta(D_{3})$ as a circulant determinant of $A_{h}$. 

\begin{thm}\label{thm:1.1.4}
Let $G$ be $D_{3}$. 
Then, we have 
$$
\Theta(D_{3}) = \prod_{\chi \in \widehat{\langle a \rangle}} \sum_{g \in \langle a \rangle} \chi(g) A_{g}. 
$$
\end{thm}

Furthermore, we obtain the analog of Theorem~$\ref{thm:1.1.1}$ on the group algebra for $D_{3}$. 

\begin{thm}\label{thm:1.1.5}
Let $G$ be $D_{3}$, 
$\omega$ be a primitive third roots of unity and $\chi_{2}$ be the non-trivial degree one representation. 
Then, we have 
\begin{align*}
\Theta(D_{3}) e 
&= \alpha_{1} 
\alpha_{2} 
\left( A_{e} e + \omega A_{a} a + \omega^{2} A_{a^{2}} a^{2} \right) 
\left( A_{e} e + \omega^{2} A_{a} a + \omega A_{a^{2}} a^{2} \right) \\
&= \prod_{\chi \in \widehat{\langle a \rangle}} \sum_{g \in \langle a \rangle} \chi(g) A_{g} g, 
\end{align*}
where 
\begin{align*}
\alpha_{1} = \sum_{g \in D_{3}} x_{g} g , \quad 
\alpha_{2} = \sum_{g \in D_{3}} \chi_{2}(g) x_{g^{-1}} g. 
\end{align*}
\end{thm}

Theorem~$\ref{thm:1.1.1}$ for $D_{3}$ and Theorem $\ref{thm:1.1.4}$ follow by applying $F$ to Theorem~$\ref{thm:1.1.5}$. 
Moreover, we obtain the following formula for inverse elements in the group algebra for $D_{3}$. 

\begin{cor}\label{cor:1.1.6}
Let $G$ be $D_{3}$. 
If $\Theta(G) \neq 0$, 
then we have 
$$
\alpha_{1}^{-1} 
= 
\frac{1}{\Theta(D_{3})} \alpha_{2} 
\left( A_{e} e + \omega A_{a} a + \omega^{2} A_{a^{2}} a^{2} \right) 
\left( A_{e} e + \omega^{2} A_{a} a + \omega A_{a^{2}} a^{2} \right). 
$$
\end{cor}
We can extend the results in Section $1.2$ into dihedral groups and generalized quaternion groups.

\subsection{Algebraic properties of irreducible factors corresponding to degree one representations}
Let $G$ be a dihedral group $D_{m}$ or generalized quaternion group $Q_{m}$. 
The commutators of irreducible factors of the factorization of the group determinant on the group algebra corresponding to degree one representations have 
the following algebraic properties. 

\begin{thm}\label{thm:1.1.7}
Let $G$ be $D_{m}$ or $Q_{m}$. 
Then, we have 
\begin{align*}
[ \alpha_{1}, \alpha_{2} ] = 0, \quad 
[ \alpha_{3}, \alpha_{4} ] = 0, 
\end{align*}
\begin{align*}
[ \alpha_{1}, \alpha_{3} + \alpha_{4} ] = 0, \quad 
[ \alpha_{2}, \alpha_{3} + \alpha_{4} ] = 0, \quad 
[ \alpha_{3}, \alpha_{1} + \alpha_{2} ] = 0, \quad 
[ \alpha_{4}, \alpha_{1} + \alpha_{2} ] = 0. 
\end{align*}
where $\alpha_{i} (1 \leq i \leq 4)$ is the irreducible factor of the factorization of the group determinant on the group algebra corresponding to the 
degree one representation. 
However, the equations for $\alpha_{3}$ and $\alpha_{4}$ do not hold 
when $G = D_{m}$ and $m$ is odd. 
\end{thm}

We remark that when $m$ is odd the number of degree one representation of $D_{m}$ is two and 
when $m$ is even the number is four. 
The number of degree one representation of $Q_{m}$ is always four. 
From Theorem $\ref{thm:1.1.7}$, we know that the degree one representations form natural pairing $(\chi_{1}, \chi_{2})$ and $(\chi_{3}, \chi_{4})$.

\subsection{Future works}
From the above results and Theorem $\ref{thm:1.1.1}$, 
we expect to find determinant formulas for matrices with elements on the group algebra.

\section{\bf{Factorization of group determinant}}
In this section, we recall the group determinant and the factorization of the group determinant.

\subsection{Factorization of group determinant}
Let $G$ be a finite group of order $n$ and $\{ x_{g} \: \vert \: g \in G \}$ be independent commuting variables. 
The group determinant $\Theta(G)$ is the determinant of the $n \times n$ matrix $\left( x_{g, h} \right)$ where $x_{g, h} = x_{g h^{-1}}$ for $g, h \in G$ 
and is thus a homogeneous polynomial of degree $n$ in $x_{g}$. 
Frobenius gave the following theorem about the factorization of the group determinant. 
\begin{thm}[Frobenius]\label{thm:2.1.1}
Let $\widehat{G}$ be a complete set of irreducible representations of $G$. 
Then, we have 
$$
\Theta(G) = \prod_{\varphi \in \widehat{G}} \det{\left( \sum_{g \in G} \varphi (g) x_{g} \right)^{\deg{\varphi}}}. 
$$
\end{thm}

\section{\bf{Results for finite abelian groups}}
In this section, we give an analog of Frobenius' theorem for finite abelian groups about factorization of the group determinant on the group algebra.

\subsection{Preparation for main results of finite abelian groups}
We prepare to explain for main results of finite abelian groups. 
Let $\{ \varphi^{(1)}, \varphi^{(2)}, \ldots, \varphi^{(s)} \}$ be a complete set of irreducible representations of G, 
$M_{k}(\mathbb{C})$ be a set of $k \times k$ matrices with entries in $\mathbb{C}$ 
and $L(G)$ be the linear space of all complex functions on $G$. 
The set $L(G)$ is a ring with addition taken pointwise and convolution as multiplication \cite[Theorem 5.2.3]{benjamin}. 
\begin{definition}[Fourier transform]\label{def:3.1.1}
Define 
$$
T: L(G) \rightarrow M_{d_{1}}(\mathbb{C}) \times M_{d_{2}}(\mathbb{C}) \times \cdots \times M_{d_{s}}(\mathbb{C})
$$
by 
$$
Tf = (\widehat{f}(\varphi^{(1)}), \widehat{f}(\varphi^{(2)}), \ldots , \widehat{f}(\varphi^{(s)})). 
$$
where
$$
\widehat{f}(\varphi^{(k)}) = \sum_{g \in G} \overline{\varphi^{(k)} (g)} f(g), 
$$
the bar denotes complex conjugation. 
We call $Tf$ the Fourier transform of $f$. 
\end{definition}
\begin{thm}[\cite{benjamin}, Theorem 5.5.6]\label{thm:3.1.2}
The Fourier transform 
$$
T: L(G) \rightarrow M_{d_{1}}(\mathbb{C}) \times M_{d_{2}}(\mathbb{C}) \times \cdots \times M_{d_{s}}(\mathbb{C}) 
$$
is an isomorphism of rings. 
\end{thm}

Let 
$$
\mathbb{C} G = \left\{ \sum_{g \in G} x_{g} g \: \vert \: x_{g} \in \mathbb{C}, g \in G \right\}
$$
be the group algebra of $G$. 

\begin{lem}\label{lem:3.1.3}
Let $G$ be a finite group. 
Then, we have 
$$
L(G) \cong \mathbb{C} G. 
$$
as an isomorphism of rigns. 
\end{lem}

\subsection{Results for finite abelian groups}
We give the following theorem that is an analog of Theorem $\ref{thm:2.1.1}$ for finite abelian groups on the group algebra.

\begin{thm}\label{thm:3.2.1}
Let $G$ be a finite abelian group and $e$ be the unit element of $G$. 
Then, we have 
$$
\Theta(G) e = \prod_{\chi \in \widehat{G}} \sum_{g \in G} \chi(g) x_{g} g. 
$$
\end{thm}
\begin{proof}
Let $G = \{g_{1}, g_{2}, \ldots, g_{n} \}$ and $\widehat{G} = \{ \chi_{1}, \chi_{2}, \ldots, \chi_{n} \}$. 
By definition, we have 
$$
\alpha = \sum_{i = 1}^{n} x_{g_{i}} g_{i} \in \mathbb{C}G. 
$$
Then, for $\chi_{k} \in \widehat{G}$, we define 
$$
\alpha_{\chi_{k}} = \sum_{i = 1}^{n} \chi_{k}(g_{i}) x_{g_{i}} g_{i} \in \mathbb{C} G, \quad 
\alpha_{\chi_{k}}^{*} = \sum_{i = 1}^{n} \chi_{k}(g_{i}) x_{g_{i}} \in \mathbb{C}. 
$$
We can regard $\alpha_{\chi_{k}}$ as an element in $L(G)$ by Lemma $\ref{lem:3.1.3}$, then we have 
$$
\alpha_{\chi_{k}} (g_{i}) = \chi_{k}(g_{i}) x_{g_{i}}. 
$$
The $l$-th component of the Fourier transform $T$ of $\alpha_{\chi_{k}}$ is 
\begin{align*}
\widehat{\alpha_{\chi_{k}}} (\chi_{l}) 
&= \sum_{i = 1}^{n} \overline{\chi_{l}(g_{i})} \alpha_{\chi_{k}} (g_{i}) \\ 
&= \sum_{i = 1}^{n} \overline{\chi_{l}(g_{i})} \chi_{k}(g_{i}) x_{g_{i}} \\ 
&= \sum_{i = 1}^{n} (\chi_{k} (\chi_{l})^{-1}) (g_{i}) x_{g_{i}} \\ 
&= \alpha_{\chi_{k} \chi_{l}^{-1}}^{*}. 
\end{align*}
Therefore, by Theorem $\ref{thm:2.1.1}$, we have 
\begin{align*}
T(\alpha_{\chi_{1}}) T(\alpha_{\chi_{2}}) \cdots T(\alpha_{\chi_{n}}) 
&= \alpha_{\chi_{1}}^{*} \alpha_{\chi_{2}}^{*} \cdots \alpha_{\chi_{n}}^{*}(1, 1, \ldots, 1) \\ 
&= \Theta(G) (1, 1, \ldots, 1). 
\end{align*}
By Theorem $\ref{thm:3.1.2}$, we have
$$
\alpha_{\chi_{1}} \alpha_{\chi_{2}} \cdots \alpha_{\chi_{n}} = \Theta(G) e. 
$$
\end{proof}
Theorem~$\ref{thm:3.2.1}$ is stronger than Theorem~$\ref{thm:2.1.1}$ for finite abelian groups. 
In fact, let $F: \mathbb{C} G \rightarrow \mathbb{C}$ be a $\mathbb{C}$-linear ring homomorphism such that $F(g) = 1$ for all $g \in G$, 
Theorem~$\ref{thm:2.1.1}$ for finite abelian groups follows by applying $F$ to Theorem~$\ref{thm:3.2.1}$. 
Moreover, we obtain the following formula for inverse elements in the group algebra. 

\begin{cor}\label{cor:3.2.2}
Let $G$ be a finite abelian group and $\chi_{1}$ be the trivial representation of $G$. 
If $\Theta(G) \neq 0$, then we have 
$$
\left( \sum_{g \in G} x_{g} g \right)^{-1} = \frac{1}{\Theta(G)} \prod_{\chi \in \widehat{G} \setminus \{ \chi_{1} \} } \left( \sum_{g \in G} \chi (g) x_{g} g \right). 
$$
\end{cor}

\section{\bf{Recall dihedral groups and generalized quaternion groups}}
In this section, we recall dihedral groups and generalized quaternion groups. 

\subsection{Recall dihedral groups}
We recall dihedral group $D_{m}$ given by the presentation 
$$
D_{m} = \langle a, b \: \vert \: a^{m} = e, b^{2} = e, b^{-1} a b = a^{-1} \rangle. 
$$
\begin{lem}\label{lem:4.1.1}
Every element $g \in D_{m}$ can be written uniquely as $g = a^{k} b^{l}$ where $0 \leq k < m$, and $l = 0, 1$. 
Namely, the order of $D_{m}$ is $2m$. 
\end{lem}

We have the following list of irreducible representations of $D_{m}$(\cite[Theorem 3]{Augmentation quotients for complex representation rings of dihedral groups}). 
Let $\omega$ be a primitive $m$-th roots of unity, for $1 \leq k \leq m - 1$. 
\clearpage 
\begin{enumerate} 
\item When $m$ is odd, and $1 \leq l \leq \frac{m-1}{2}$. 
\begin{table}[h]
\begin{center}
\begin{tabular}{|l|c|c|c|c|} \hline 
& $e$ & $a^{k}$ & $b$ & $a^{k}b$ \\ \hline
$\chi_{1}$ & $1$ & $1$ & $1$ & $1$ \\
$\chi_{2}$ & $1$ & $1$ & $-1$ & $-1$ \\
$\varphi_{l}$ 
& $\begin{bmatrix} $1$ & $0$ \\ $0$ & $1$ \end{bmatrix}$ 
& $\begin{bmatrix} \omega^{lk} & $0$ \\ $0$ & \omega^{-lk} \end{bmatrix}$ 
& $\begin{bmatrix} $0$ & $1$ \\ $1$ & $0$ \end{bmatrix}$ 
& $\begin{bmatrix} $0$ & \omega^{lk} \\ \omega^{-lk} & $0$ \end{bmatrix}$ 
\\ \hline
\end{tabular}
\end{center}
\end{table}
\item When $m$ is even, and $1 \leq l \leq \frac{m}{2}-1$. 
\begin{table}[h]
\begin{center}
\begin{tabular}{|l|c|c|c|c|} \hline 
& $e$ & $a^{k}$ & $b$ & $a^{k}b$ \\ \hline
$\chi_{1}$ & $1$ & $1$ & $1$ & $1$ \\
$\chi_{2}$ & $1$ & $1$ & $-1$ & $-1$ \\
$\chi_{3}$ & $1$ & $(-1)^{k}$ & $1$ & $(-1)^{k}$ \\
$\chi_{4}$ & $1$ & $(-1)^{k}$ & $-1$ & $(-1)^{k+1}$ \\
$\varphi_{l}$ 
& $\begin{bmatrix} $1$ & $0$ \\ $0$ & $1$ \end{bmatrix}$ 
& $\begin{bmatrix} \omega^{lk} & $0$ \\ $0$ & \omega^{-lk} \end{bmatrix}$ 
& $\begin{bmatrix} $0$ & $1$ \\ $1$ & $0$ \end{bmatrix}$ 
& $\begin{bmatrix} $0$ & \omega^{lk} \\ \omega^{-lk} & $0$ \end{bmatrix}$ 
\\ \hline
\end{tabular}
\end{center}
\end{table}
\end{enumerate}

\subsection{Recall generalized quaternion groups}
We recall generalized quaternion group $Q_{m}$ given by the presentation 
$$
Q_{m} = \langle a, b \: \vert \: a^{2m} = e, b^{2} = a^{m}, b^{-1} a b = a^{-1} \rangle. 
$$
\begin{lem}\label{lem:4.2.1}
Every element $g \in Q_{m}$ can be written uniquely as $g = a^{k} b^{l}$ where $0 \leq k < 2m$, and $l = 0, 1$. 
Namely, the order of $Q_{m}$ is $4m$. 
\end{lem}

We have the following list of irreducible representations of $Q_{m}$. 
Let $\omega$ be a primitive $2m$-th roots of unity, for $1 \leq k \leq 2m - 1$. 
\begin{enumerate}
\item When $m$ is odd, and $1 \leq l \leq m - 1$. 
\begin{table}[h]
\begin{center}
\begin{tabular}{|l|c|c|c|c|} \hline 
& $e$ & $a^{k}$ & $b$ & $a^{k}b$ \\ \hline
$\chi_{1}$ & $1$ & $1$ & $1$ & $1$ \\
$\chi_{2}$ & $1$ & $1$ & $-1$ & $-1$ \\
$\chi_{3}$ & $1$ & $(-1)^{k}$ & $i$ & $i (-1)^{k}$ \\ 
$\chi_{4}$ & $1$ & $(-1)^{k}$ & $-i$ & $i (-1)^{k+1}$ \\ 
$\varphi_{l}$ 
& $\begin{bmatrix} 1 & 0 \\ 0 & 1 \end{bmatrix}$ 
& $\begin{bmatrix} \omega^{lk} & 0 \\ 0 & \omega^{-lk} \end{bmatrix}$ 
& $\begin{bmatrix} 0 & 1 \\ -1 & 0 \end{bmatrix}$ 
& $\begin{bmatrix} 0 & \omega^{lk} \\ - \omega^{-lk} & 0 \end{bmatrix}$ 
\\ \hline
\end{tabular}
\end{center}
\end{table}
\item When $m$ is even, and $1 \leq l \leq m - 1$. 
\begin{table}[h]
\begin{center}
\begin{tabular}{|l|c|c|c|c|} \hline 
& $e$ & $a^{k}$ & $b$ & $a^{k}b$ \\ \hline
$\chi_{1}$ & $1$ & $1$ & $1$ & $1$ \\
$\chi_{2}$ & $1$ & $1$ & $-1$ & $-1$ \\
$\chi_{3}$ & $1$ & $(-1)^{k}$ & $1$ & $(-1)^{k}$ \\
$\chi_{4}$ & $1$ & $(-1)^{k}$ & $-1$ & $(-1)^{k+1}$ \\
$\varphi_{l}$ 
& $\begin{bmatrix} 1 & 0 \\ 0 & 1 \end{bmatrix}$ 
& $\begin{bmatrix} \omega^{lk} & 0 \\ 0 & \omega^{-lk} \end{bmatrix}$ 
& $\begin{bmatrix} 0 & 1 \\ -1 & 0 \end{bmatrix}$ 
& $\begin{bmatrix} 0 & \omega^{lk} \\ - \omega^{-lk} & 0 \end{bmatrix}$ 
\\ \hline
\end{tabular}
\end{center}
\end{table}
\end{enumerate}

\section{\bf{Main results for dihedral groups and generalized quaternion groups}}
In this section, 
we describe the group determinant of dihedral groups and generalized quaternion groups as a circulant determinant of homogeneous polynomials and
we give an analog of Frobenius' theorem for dihedral groups and generalized quaternion groups about factorization of the group determinant on the group algebra.

\subsection{Describe some group determinant as a circulant determinant}
Let $G$ be $D_{m}$ or $Q_{m}$, 
$\langle a \rangle$ be the subgroup generated by $\{ a \}$ and define 
$$
\chi_{l}' (a^{k}) = \omega^{l k} 
$$
for $0 \leq l \leq \vert \langle a \rangle \vert - 1 $ where $\vert H \vert$ is the order of the group $H$. 
Notice that $\chi_{l}'$ is a degree one representation of $\langle a \rangle$. 
Let $A_{h}$ be a homogeneous polynomial defined by 
$$
A_{h} = \sum_{g \in \langle a \rangle} \left( x_{g} x_{h g} - x_{g b} x_{h g b^{-1}} \right) 
$$
for $h \in G$. 
This homogeneous polynomial $A_{h}$ has the following properties. 

\begin{lem}\label{lem:5.1.1}
Let $G$ be $D_{m}$ or $Q_{m}$. 
Then, we have 
\begin{enumerate}
\item If h $\in G \setminus \langle a \rangle,\ A_{h} = 0$. 
\item For all $h \in G,\ A_{h} = A_{h^{-1}}$. 
\end{enumerate}
\end{lem}
\begin{proof}
Element $h \in G \setminus \langle a \rangle$ can be written as $a^{k} b$ for $0 \leq k \leq \vert \langle a \rangle \vert - 1$. 
If $g \in \langle a \rangle$, then $bg = g^{-1} b$ and we can write $g$ as $a^{k} g^{-1}$. 
From these, we have 
\begin{align*}
A_{a^{k} b} 
&= \sum_{g \in \langle a \rangle} x_{g} x_{a^{k} b g} - \sum_{g \in \langle a \rangle} x_{g b} x_{a^{k} b g b^{-1}} \\ 
&= \sum_{g \in \langle a \rangle} x_{a^{k} g^{-1}} x_{a^{k} b a^{k} g^{-1}} - \sum_{g \in \langle a \rangle} x_{g b} x_{a^{k} g^{-1}} \\ 
&= \sum_{g \in \langle a \rangle} x_{a^{k} g^{-1}} x_{g b} - \sum_{g \in \langle a \rangle} x_{a^{k} g^{-1}} x_{g b} \\
&= 0. 
\end{align*}
This proves the first claim. 
If $h \in G \setminus \langle a \rangle$, 
then, we have $h^{-1} \in G \setminus \langle a \rangle$. 
(2) follows from first claim. 
We assume that $h \in \langle a \rangle$. 
We can write $g$ as $h^{-1}g$ in the first sum and $g$ as $g = h^{-1} g b^{-2}$ in the second.
From this and $b^{-3} = b$, we have 
\begin{align*}
A_{h} 
&= \sum_{g \in \langle a \rangle} x_{g} x_{h g} - \sum_{g \in \langle a \rangle} x_{g b} x_{h g b^{-1}} \\ 
&= \sum_{g \in \langle a \rangle} x_{(h^{-1} g)} x_{h (h^{-1} g)} - \sum_{g \in \langle a \rangle} x_{(h^{-1} g b^{-2}) b} x_{h (h^{-1} g b^{-2}) b^{-1}} \\ 
&= \sum_{g \in \langle a \rangle} x_{h^{-1} g} x_{g} - \sum_{g \in \langle a \rangle} x_{h^{-1} g b^{-1}} x_{g b^{-3}} \\ 
&= \sum_{g \in \langle a \rangle} x_{g} x_{h^{-1} g} - \sum_{g \in \langle a \rangle} x_{h^{-1} g b^{-1}} x_{g b} \\ 
&= A_{h^{-1}}. 
\end{align*}
This completes the proof. 
\end{proof}

\begin{lem}\label{lem:5.1.2}
Let $G$ be $D_{m}$ or $Q_{m}$ and $\chi_{l} (1 \leq l \leq 4)$ be a degree one representation of Section 4. 
Then, we have 
\begin{align*}
\sum_{g \in G} \chi_{1}(g) x_{g} \sum_{g' \in G} \chi_{2}(g') x_{g'} &= \sum_{h \in \langle a \rangle} A_{h}, \\ 
\sum_{g \in G} \chi_{3}(g) x_{g} \sum_{g' \in G} \chi_{4}(g') x_{g'} &= \sum_{h \in \langle a \rangle} \chi_{\frac{\vert \langle a \rangle \vert}{2}}'(h) A_{h}. 
\end{align*}
\end{lem}

This lemma will be proved later. 

\begin{lem}\label{lem:5.1.3}
Let $G$ be $D_{m}$ or $Q_{m}$ and $\varphi_{l}$ be degree two representations of Section~4. 
Then, we have 
$$
\det{\left( \sum_{g \in G} \varphi_{l}(g) x_{g} \right) } = \sum_{h \in \langle a \rangle} \chi_{l}'(h) A_{h}. 
$$
\end{lem}
\begin{proof}
Define the function $\delta$ by 
\begin{align*}
\delta(G) = 
\begin{cases} 
1 & G =D_{m}, \\ 
-1 & G = Q_{m}. 
\end{cases} 
\end{align*} 
We obtain 
\begin{align}
\det{ \left( \sum_{g \in G} \varphi_{l}(g) x_{g} \right) } 
&= \det{ \left( \sum_{g \in \langle a \rangle} \varphi_{l}(g) x_{g} + \sum_{g \in G \setminus \langle a \rangle} \varphi_{l}(g) x_{g} \right) } \nonumber \\ 
&= \det{ \left( \sum_{g \in \langle a \rangle} \varphi_{l}(g) x_{g} + \sum_{g \in \langle a \rangle} \varphi_{l}(g b) x_{g b} \right) } \nonumber \\ 
&= \det{ 
\begin{bmatrix} 
\displaystyle\sum_{g \in \langle a \rangle} \chi_{l}'(g) x_{g} & \displaystyle\sum_{g \in \langle a \rangle} \chi_{l}'(g) x_{g b} \\ 
\delta(G) \displaystyle\sum_{g \in \langle a \rangle} \chi_{l}'(g^{-1}) x_{g b} & \displaystyle\sum_{g \in \langle a \rangle} \chi_{l}'(g^{-1}) x_{g}
\end{bmatrix} } \nonumber \\ 
&= \sum_{h \in \langle a \rangle} \sum_{g \in \langle a \rangle} \chi_{l}'(h g^{-1}) x_{h} x_{g} 
- \delta(G) \sum_{g \in \langle a \rangle} \sum_{h \in \langle a \rangle} \chi_{l}'(g^{-1} h) x_{g b} x_{h b}. \label{eq:5.1.3-1}
\end{align}

We can write $h$ as $h g$ in the first sum and $h$ as $h g b^{-2}$ in the second. 
Then~$(\ref{eq:5.1.3-1})$ equals 
\begin{align} 
&\sum_{h \in \langle a \rangle} \chi_{l}'(h) \sum_{g \in \langle a \rangle} x_{g} x_{h g} - 
\delta(G) \sum_{g \in \langle a \rangle} \sum_{h \in \langle a \rangle} \chi_{l}'(g^{-1} (h g b^{-2})) x_{g b} x_{(h g b^{-2}) b} \nonumber \\ 
= &\sum_{h \in \langle a \rangle} \chi_{l}'(h) \sum_{g \in \langle a \rangle} x_{g} x_{h g} - 
\delta(G) \chi_{l}'(b^{-2}) \sum_{h \in \langle a \rangle} \chi_{l}'(h) \sum_{g \in \langle a \rangle} x_{g b} x_{h g b^{-1}}. \label{eq:5.1.3-2}
\end{align}
It is easy to see that $\delta(G) \chi_{l}'(b^{-2}) = 1$. 
Then, $(\ref{eq:5.1.3-2})$ equals 
\begin{align*} 
\sum_{h \in \langle a \rangle} \chi_{l}'(h) \left( \sum_{g \in \langle a \rangle} (x_{g} x_{h g} - x_{g b} x_{h g b^{-1}}) \right) 
= \sum_{h \in \langle a \rangle} \chi_{l}'(h) A_{h}
\end{align*}
as required. 
\end{proof}

\begin{lem}\label{lem:5.1.4}
Let $\chi'$ be a degree one representation of $\langle a \rangle$. 
Then, we have 
$$
\sum_{g \in \langle a \rangle} \chi'(g) A_{g} = \sum_{g \in \langle a \rangle} \overline{\chi'(g)} A_{g}. 
$$
\end{lem}
\begin{proof}
This follows from Lemma $\ref{lem:5.1.1}$. 
\end{proof}

Let $\widehat{H}$ be a complete set of irreducible representations of the group $H$. 
We describe $\Theta(G)$ as a circulant determinant of $A_{h}$. 

\begin{thm}\label{thm:5.1.5}
Let $G$ be $D_{m}$ or $Q_{m}$. 
Then, we have 
$$
\Theta(G) = \prod_{\chi' \in \widehat{\langle a \rangle}} \sum_{g \in \langle a \rangle} \chi'(g) A_{g}. 
$$
\end{thm}
\begin{proof}
First, let $G = D_{m}$ and $m$ be odd. 
By Theorem $\ref{thm:2.1.1}$ and Lemma $\ref{lem:5.1.2}$, we have 
\begin{align*}
\Theta(G) 
&= \left( \sum_{g \in \langle a \rangle} A_{g} \right) \prod_{l = 1}^{\frac{m - 1}{2}} \left( \sum_{g \in \langle a \rangle} \chi_{l}'(g) A_{g} \right)^{2}. 
\end{align*}
From Lemma $\ref{lem:5.1.4}$, we have 
\begin{align*}
\Theta(G) 
&= \left( \sum_{g \in \langle a \rangle} A_{g} \right) 
\prod_{l = 1}^{\frac{m - 1}{2}} \left( \sum_{g \in \langle a \rangle} \chi_{l}'(g) A_{g} \right) 
\left( \sum_{g \in \langle a \rangle} \chi_{m - l}'(g) A_{g} \right) \\ 
&= \prod_{l = 0}^{m - 1} \sum_{g \in \langle a \rangle} \chi_{l}'(g) A_{g}. 
\end{align*}
In other case, by Theorem $\ref{thm:2.1.1}$ and Lemma $\ref{lem:5.1.2}$, we have 
\begin{align*}
\Theta(G) 
&= \left( \sum_{g \in \langle a \rangle} A_{g} \right)
\left( \sum_{g \in \langle a \rangle} \chi_{\frac{\vert \langle a \rangle \vert}{2}}'(g) A_{g} \right) 
\prod_{l = 1}^{\frac{\vert \langle a \rangle \vert}{2} - 1} \left( \sum_{g \in \langle a \rangle} \chi_{l}'(g) A_{g} \right)^{2}. 
\end{align*}
From Lemma $\ref{lem:5.1.4}$, we have 
\begin{eqnarray*}
\Theta(G) 
& = & \left( \sum_{g \in \langle a \rangle} A_{g} \right) 
\left( \sum_{g \in \langle a \rangle} \chi_{\frac{\vert \langle a \rangle \vert}{2}}'(g) A_{g} \right) \\ 
& & \prod_{l = 1}^{\frac{\vert \langle a \rangle \vert}{2} - 1} \left( \sum_{g \in \langle a \rangle} \chi_{l}'(g) A_{g} \right) 
\left( \sum_{g \in \langle a \rangle} \chi_{\vert \langle a \rangle \vert - l}'(g) A_{g} \right) \\ 
& = & \prod_{l = 0}^{\vert \langle a \rangle \vert - 1} \sum_{g \in \langle a \rangle} \chi_{l}'(g) A_{g}. 
\end{eqnarray*}
\end{proof}

\subsection{Main results for dihedral groups and generalized quaternion groups}
Let $G$ be $D_{m}$ or $Q_{m}$ and $\chi_{l} (1 \leq l \leq 4)$ be a degree one representation of Section 4. 
By definition, we have 
\begin{align*}
\alpha_{1} = \sum_{g \in G} x_{g} g, \quad 
\alpha_{2} = \sum_{g \in \langle a \rangle} \chi_{2}(g) x_{g^{-1}} g + \sum_{g \in G \setminus \langle a \rangle} \chi_{2}(g) x_{g} g, \\
\alpha_{3} = \sum_{g \in G} \chi_{3}(g) x_{g} g, \quad 
\alpha_{4} = \sum_{g \in \langle a \rangle} \chi_{4}(g) x_{g^{-1}} g + \sum_{g \in G \setminus \langle a \rangle} \chi_{4}(g) x_{g} g. 
\end{align*}

\begin{lem}\label{lem:5.2.1}
The following formula holds. 
$$
\alpha_{1} \alpha_{2} = \sum_{h \in \langle a \rangle} A_{h} h. 
$$
\end{lem}
\begin{proof}
We have 
\begin{align}
\alpha_{1} \alpha_{2} 
&= \sum_{g' \in G} x_{g'} g' \sum_{g \in \langle a \rangle} \chi_{2}(g) x_{g^{-1}} x_{g^{-1}} g + 
\sum_{g' \in G} x_{g'} g' \sum_{g \in G \setminus \langle a \rangle} \chi_{2}(g) x_{g} g \nonumber \\ 
&= \sum_{g' \in G} \sum_{g \in \langle a \rangle} \chi_{2}(g) x_{g'} x_{g^{-1}} g' g + 
\sum_{g' \in G} \sum_{g \in G \setminus \langle a \rangle} \chi_{2}(g) x_{g'} x_{g} g' g \nonumber \\ 
&= \sum_{g' \in G} \sum_{g \in \langle a \rangle} x_{g'} x_{g^{-1}} g' g + 
\sum_{g' \in G} \sum_{g \in \langle a \rangle} \chi_{2}(g b) x_{g'} x_{g b} g' g b \nonumber \\ 
&= \sum_{g' \in G} \sum_{g \in \langle a \rangle} x_{g'} x_{g^{-1}} g' g - 
\sum_{g' \in G} \sum_{g \in \langle a \rangle} x_{g'} x_{g b} g' g b \nonumber \\ 
&= \sum_{g' \in G} \sum_{g \in \langle a \rangle} x_{g'} x_{g^{-1}} g' g - 
\sum_{g' \in G} \sum_{g \in \langle a \rangle} x_{g'} x_{g b} g' b g^{-1} \nonumber \\ 
&= \sum_{g' \in G} \sum_{g \in \langle a \rangle} x_{g'} x_{g^{-1}} g' g - 
\sum_{g' \in G} \sum_{g \in \langle a \rangle} x_{g' b^{-1}} x_{g b} g' g^{-1} \nonumber \\ 
&= \sum_{g' \in G} \sum_{g \in \langle a \rangle} x_{g'} x_{g^{-1}} g' g - 
\sum_{g' \in G} \sum_{g \in \langle a \rangle} x_{g' b^{-1}} x_{g^{-1} b} g' g. \label{eq:5.2.1}
\end{align}
We can make the change of variables $h = g' g$ to compute, then (\ref{eq:5.2.1}) equals 
\begin{align*}
&\sum_{h \in G} \sum_{g \in \langle a \rangle} x_{h g^{-1}} x_{g^{-1}} h - 
\sum_{h \in G} \sum_{g \in \langle a \rangle} x_{h g^{-1} b^{-1}} x_{g^{-1} b} h \\ 
= &\sum_{h \in G} \sum_{g \in \langle a \rangle} x_{h g} x_{g} h - 
\sum_{h \in G} \sum_{g \in \langle a \rangle} x_{h g b^{-1}} x_{g b} h \\ 
= &\sum_{h \in G} \sum_{g \in \langle a \rangle} (x_{g} x_{hg} - x_{g b} x_{h g b^{-1}}) h \\ 
= &\sum_{h \in G} A_{h} h \\ 
= &\sum_{h \in \langle a \rangle} A_{h} h
\end{align*}
as required. 
\end{proof}

\begin{lem}\label{lem:5.2.2}
The following formula holds. 
$$
\alpha_{3} \alpha_{4} = \sum_{h \in \langle a \rangle} \chi_{\frac{\vert \langle a \rangle \vert}{2}}'(h) A_{h} h. 
$$
\end{lem}
\begin{proof}
We have 
\begin{align*}
\alpha_{3} \alpha_{4} 
&= \sum_{g' \in G} \chi_{3}(g') x_{g'} g' \sum_{g \in \langle a \rangle} \chi_{4}(g) x_{g^{-1}} g + 
\sum_{g' \in G} \chi_{3}(g') x_{g'} g' \sum_{g \in G \setminus \langle a \rangle} \chi_{4}(g) x_{g} g \\ 
&= \sum_{g' \in G} \sum_{g \in \langle a \rangle} \chi_{3}(g') \chi_{4}(g) x_{g'} x_{g^{-1}} g' g + 
\sum_{g' \in G} \sum_{g \in G \setminus \langle a \rangle} \chi_{3}(g') \chi_{4}(g) x_{g'} x_{g} g' g. 
\end{align*}
If $g \in \langle a \rangle$, then, we have $\chi_{3}(g) \chi_{4}(g^{-1}) = 1$. 
We have 
\begin{align*}
\sum_{g' \in G} \sum_{g \in \langle a \rangle} \chi_{3}(g') \chi_{4}(g) x_{g'} x_{g^{-1}} g' g 
&= \sum_{h \in G} \sum_{g \in \langle a \rangle} \chi_{3}(h g^{-1}) \chi_{4}(g) x_{h g^{-1}} x_{g^{-1}} h \\ 
&= \sum_{h \in G} \sum_{g \in \langle a \rangle} \chi_{3}(h g) \chi_{4}(g^{-1}) x_{h g} x_{g} h \\ 
&= \sum_{h \in G} \sum_{g \in \langle a \rangle} \chi_{3}(h) x_{g} x_{h g} h. 
\end{align*}
We have 
\begin{align}
\sum_{g' \in G} \sum_{g \in G \setminus \langle a \rangle} \chi_{3}(g') \chi_{4}(g) x_{g'} x_{g} g' g 
&= \sum_{g' \in G} \sum_{g \in \langle a \rangle} \chi_{3}(g') \chi_{4}(g b) x_{g'} x_{g b} g' g b \nonumber \\ 
&= \chi_{4}(b) \sum_{g' \in G} \sum_{g \in \langle a \rangle} \chi_{3}(g') \chi_{4}(g) x_{g'} x_{g b} g' b g^{-1} \nonumber \\ 
&= \chi_{4}(b) \sum_{g' \in G} \sum_{g \in \langle a \rangle} \chi_{3}(g' b^{-1}) \chi_{4}(g) x_{g' b^{-1}} x_{g b} g' g^{-1}. \label{eq:5.2.2} 
\end{align}
It is easy that to see $\chi_{3}(b^{-1}) \chi_{4}(b) = -1$, then $(\ref{eq:5.2.2})$ equals 
\begin{align}
 &- \sum_{g' \in G} \sum_{g \in \langle a \rangle} \chi_{3}(g') \chi_{4}(g) x_{g' b^{-1}} x_{g b} g' g^{-1} \nonumber \\ 
= &- \sum_{h \in G} \sum_{g \in \langle a \rangle} \chi_{3}(h g^{-1}) \chi_{4}(g^{-1}) x_{h g^{-1} b^{-1}} x_{g^{-1} b} h \nonumber \\ 
= &- \sum_{h \in G} \sum_{g \in \langle a \rangle} \chi_{3}(h g) \chi_{4}(g) x_{h g b^{-1}} x_{g b} h. \label{eq:5.2.2.2}
\end{align}
If $g \in \langle a \rangle$, then we have $\chi_{3}(g) \chi_{4}(g) = 1$ and $(\ref{eq:5.2.2.2})$ equals 
\begin{align*}
 - \sum_{h \in G} \sum_{g \in \langle a \rangle} \chi_{3}(h) x_{h g b^{-1}} x_{g b} h. 
\end{align*}
We compute 
\begin{align*}
\alpha_{3} \alpha_{4} 
&= \sum_{h \in G} \chi_{3}(h) \sum_{g \in \langle a \rangle} (x_{g} x_{h g} - x_{g b} x_{h g b^{-1}}) h \\ 
&= \sum_{h \in G} \chi_{3}(h) A_{h} h \\ 
&= \sum_{h \in \langle a \rangle} \chi_{3}(h) A_{h} h \\
&= \sum_{h \in \langle a \rangle} \chi_{\frac{\vert \langle a \rangle \vert}{2}}'(h) A_{h} h
\end{align*}
as required. 
\end{proof}

Lemmas~$\ref{lem:5.2.1}$ and~$\ref{lem:5.2.2}$ are stronger than Lemma~$\ref{lem:5.1.2}$. 
In fact, Lemma~$\ref{lem:5.1.2}$ follows by applying $F$ to Lemmas~$\ref{lem:5.2.1}$ and~$\ref{lem:5.2.2}$. 

\begin{thm}\label{thm:5.2.3}
Let $G$ be $D_{m}$ or $Q_{m}$ and $e$ be the unit element of $G$. 
Then, we have 
$$
\Theta(G) e = \prod_{\chi' \in \widehat{\langle a \rangle}} \sum_{g \in \langle a \rangle} \chi'(g) x_{g} g. 
$$
\end{thm}
\begin{proof}
The group $\langle a \rangle$ is a finite abelian group and by Theorem $\ref{thm:3.2.1}$, there exists a $C$ such that 
$$
Ce = \prod_{\chi' \in \widehat{\langle a \rangle}} \sum_{g \in \langle a \rangle} \chi'(g) x_{g} g. 
$$
By Theorem $\ref{thm:5.1.5}$ and the mapping $F$, we have 
$$
C = \Theta(G). 
$$
This completes the proof. 
\end{proof}

Theorem~$\ref{thm:2.1.1}$ for $D_{m}$ or $Q_{m}$ and Theorem~$\ref{thm:5.1.5}$ follow by applying $F$ to Theorem~$\ref{thm:5.2.3}$. 
Moreover, we obtain the following formula for inverse elements in the group algebra for $D_{m}$ and $Q_{m}$. 

\begin{cor}\label{cor:5.2.4}
Let $G$ be $D_{m}$ or $Q_{m}$. If $\Theta(G) \neq 0$, 
\begin{enumerate}
\item When $G = D_{m}$ and $m$ is odd, we have
$$
\alpha_{1}^{-1} = \frac{1}{\Theta(G)} \alpha_{2} \prod_{\chi' \in \widehat{\langle a \rangle} \setminus \{ \chi'_{0} \}} \sum_{g \in \langle a \rangle} \chi'(g) x_{g} g. 
$$
\item In other cases, we have 
$$
\alpha_{1}^{-1} = \frac{1}{\Theta(G)} \alpha_{2} \alpha_{3} \alpha_{4} \prod_{\chi' \in \widehat{\langle a \rangle} \setminus \{ \chi'_{0} \}} \sum_{g \in \langle a \rangle} \chi'(g) x_{g} g. 
$$
\end{enumerate}
\end{cor}

Note that the group algebra is non-commutative, so the order of the factors is important.

\section{\bf{Algebraic properties of irreducible factors corresponding to degree one representations}}
In this section, we follow the commutators of irreducible factors of the factorization of the group determinant on the group algebra for $D_{m}$ and $Q_{m}$ 
corresponding to degree one representations have interesting algebraic properties. 
From this result, we know that degree one representations from natural pairing. 

\subsection{Transformations of irreducible factors corresponding to degree one representations}

In this subsection, we obtain transformations of irreducible factors corresponding to degree one representations. 
We need it to prove for algebraic properties of irreducible factors corresponding to degree one representations. 

\begin{lem}\label{lem:6.1.1}
Let $G$ be $D_{m}$ or $Q_{m}$. 
Under the change of variable 
$$
x_{g} \mapsto \begin{cases} 
\chi_{2}(g) x_{g^{-1}} & {\rm if} \quad g \in \langle a \rangle, \\ 
\chi_{2}(g) x_{g} & \rm{otherwise}. 
\end{cases}
$$
Then $\alpha_{1}$ becomes $\alpha_{2}$ by and vice versa, 
the same holds for $\alpha_{3}$ and $\alpha_{4}$. 
\end{lem}
\begin{proof}
From the definition of $\alpha_{1}$, 
the factor $\alpha_{1}$ becomes $\alpha_{3}$ under the change of variable. 
The factor $\alpha_{2}$ becomes 
\begin{align*}
\sum_{g \in \langle a \rangle} \chi_{2}(g) \chi_{2}(g^{-1}) x_{g} g + \sum_{g \in G \setminus \langle a \rangle} \chi_{2}(g) \chi_{2}(g) x_{g} g 
&= \sum_{g \in \langle a \rangle} x_{g} g + \chi_{2}(b)^{2} \sum_{g \in G \setminus \langle a \rangle} x_{g} g \\ 
&= \sum_{g \in \langle a \rangle} x_{g} g + \sum_{g \in G \setminus \langle a \rangle} x_{g} g \\ 
&= \sum_{g \in G} x_{g} g \\ 
&= \alpha_{1}. 
\end{align*}
The factor $\alpha_{3}$ becomes 
\begin{align*}
& \sum_{g \in \langle a \rangle} \chi_{3}(g) \chi_{2}(g) x_{g^{-1}} g + \sum_{g \in G \setminus \langle a \rangle} \chi_{3}(g) \chi_{2}(g) x_{g} g \\ 
&\quad = \sum_{g \in \langle a \rangle} \chi_{3}(g) x_{g^{-1}} g - \sum_{g \in G \setminus \langle a \rangle} \chi_{3}(g) x_{g} g \\ 
&\quad = \sum_{g \in \langle a \rangle} \chi_{4}(g) x_{g^{-1}} g + \sum_{g \in G \setminus \langle a \rangle} \chi_{4}(g) x_{g} g\\ 
&\quad = \alpha_{4}. 
\end{align*}
The factor $\alpha_{4}$ becomes 
\begin{align*}
& \sum_{g \in \langle a \rangle} \chi_{4}(g) \chi_{2}(g^-1{}) x_{g} g + \sum_{g \in G \setminus \langle a \rangle} \chi_{4}(g) \chi_{2}(g) x_{g} g \\ 
&\quad = \sum_{g \in \langle a \rangle} \chi_{4}(g) x_{g} g - \sum_{g \in G \setminus \langle a \rangle} \chi_{4}(g) x_{g} g \\ 
&\quad = \sum_{g \in \langle a \rangle} \chi_{3}(g) x_{g} g + \sum_{g \in G \setminus \langle a \rangle} \chi_{3}(g) x_{g} g \\ 
&\quad = \alpha_{3}. 
\end{align*}
This completes the proof. 
\end{proof}

\begin{lem}\label{lem:6.1.2}
Let $G$ be $D_{m}$ or $Q_{m}$. 
Under the change of variable $x_{g} \mapsto \chi_{3}(g) x_{g}$, 
the factor $\alpha_{1}$ becomes $\alpha_{3}$, 
the factor $\alpha_{2}$ becomes $\alpha_{4}$, and 
$\alpha_{3}$ becomes 
$$
\sum_{g \in \langle a \rangle} x_{g} g - \sum_{g \in G \setminus \langle a \rangle} x_{g} g 
$$
in the case that $m$ is odd and $\alpha_{1}$ otherwise. 
The factor $\alpha_{4}$ becomes 
$$
\sum_{g \in \langle a \rangle} x_{g^{-1}} g + \sum_{g \in G \setminus \langle a \rangle} x_{g} g
$$
in the case that $m$ is odd and $\alpha_{2}$ otherwise. 
Namely, 
the factor $\alpha_{3} + \alpha_{4}$ becomes $\alpha_{1} + \alpha_{2}$. 
\end{lem}
\begin{proof}
In the definition of $\alpha_{1}$, we obtain $\alpha_{3}$. 
The factor $\alpha_{2}$ becomes 
\begin{align*}
& \sum_{g \in \langle a \rangle} \chi_{2}(g) \chi_{3}(g^{-1}) x_{g^{-1}} g + \sum_{g \in G \setminus \langle a \rangle} \chi_{2}(g) \chi_{3}(g) x_{g} g \\ 
&\quad = \sum_{g \in \langle a \rangle} \chi_{3}(g) x_{g^{-1}} g - \sum_{g \in G \setminus \langle a \rangle} \chi_{3}(g) x_{g} g \\ 
&\quad = \sum_{g \in \langle a \rangle} \chi_{4}(g) x_{g^{-1}} g + \sum_{g \in G \setminus \langle a \rangle} \chi_{4}(g) x_{g} g \\ 
&\quad = \alpha_{4}. 
\end{align*}
The factor $\alpha_{3}$ becomes 
\begin{align*}
\sum_{g \in G} \chi_{3}(g)^{2} x_{g} g 
&= \sum_{g \in \langle a \rangle} x_{g} g + \chi_{3}(b)^{2} \sum_{g \in G \setminus \langle a \rangle} x_{g} g. 
\end{align*}
When $m$ is even, right hand side is equal to $\alpha_{1}$, otherwise is equal to 
$$
\sum_{g \in \langle a \rangle} x_{g} g - \sum_{g \in G \setminus \langle a \rangle} x_{g} g. 
$$
The factor $\alpha_{4}$ becomes 
\begin{align*}
& \sum_{g \in \langle a \rangle} \chi_{4}(g) \chi_{3}(g^{-1}) x_{g^{-1}} g + \sum_{g \in G \setminus \langle a \rangle} \chi_{4}(g) \chi_{3}(g) x_{g} g \\ 
& \quad = \sum_{g \in \langle a \rangle} x_{g^{-1}} g - \chi_{3}(b)^{2} \sum_{g \in G \setminus \langle a \rangle} x_{g} g, 
\end{align*}
When $m$ is odd, right hand side is equal to $\alpha_{2}$, otherwise is equal to
$$
\sum_{g \in \langle a \rangle} x_{g^{-1}} g + \sum_{g \in G \setminus \langle a \rangle} x_{g} g. 
$$
This completes the proof. 
\end{proof}

\subsection{Algebraic properties of irreducible factors corresponding to degree one representations}
In this subsection, we follow the commutators of irreducible factors of the factorization of the group determinant on the group algebra for $D_{m}$ and $Q_{m}$ corresponding to degree one representations have interesting algebraic properties. 
From this result, we know that degree one representations form natural pairing. 

\begin{lem}\label{lem:6.2.1}
If $h \in \langle a \rangle$, then we have 
$$
\sum_{g \in \langle a \rangle} x_{g b} x_{b^{-1} g^{-1} h} = \sum_{g \in \langle a \rangle} x_{g b} x_{h g b^{-1}}. 
$$
\end{lem}
\begin{proof}
We compute 
\begin{align*}
\sum_{g \in \langle a \rangle} x_{g b} x_{b^{-1} g^{-1} h} 
&= \sum_{g \in \langle a \rangle} x_{b^{-1} g^{-1} h} x_{g b} \\ 
&= \sum_{g \in \langle a \rangle} x_{h^{-1} g b^{-1}} x_{g b} \\ 
&= \sum_{g \in \langle a \rangle} x_{h^{-1} g b^{3}} x_{g b} \\ 
&= \sum_{g \in \langle a \rangle} x_{h^{-1} (h g b^{-2}) b^{3}} x_{h g b^{-2} b} \\ 
&= \sum_{g \in \langle a \rangle} x_{g b} x_{h g b^{-1}}
\end{align*}
as required. 
\end{proof}

\begin{lem}\label{lem:6.2.2}
If $h \in G \setminus \langle a \rangle$, then we have 
$$
\sum_{g \in G \setminus \langle a \rangle} x_{g} x_{g^{-1} h} = \sum_{g \in \langle a \rangle} x_{g} x_{g h}. 
$$
\end{lem}
\begin{proof}
We compute 
\begin{align*}
\sum_{g \in G \setminus \langle a \rangle} x_{g} x_{g^{-1} h} 
&= \sum_{g \in \langle a \rangle} x_{g b} x_{b^{-1} g^{-1} a^{k} b} \\ 
&= \sum_{g \in \langle a \rangle} x_{g b} x_{b^{-1} b g a^{-k}} \\ 
&= \sum_{g \in \langle a \rangle} x_{g b} x_{g a^{-k}} \\ 
&= \sum_{g \in \langle a \rangle} x_{g a^{k} b} x_{g} \\ 
&= \sum_{g \in \langle a \rangle} x_{g} x_{g a^{k} b} \\ 
&= \sum_{g \in \langle a \rangle} x_{g} x_{g h} 
\end{align*}
as required. 
\end{proof}

Define $[a, b] = a b - b a$ where $a, b \in \mathbb{C} G$. 
This is equal to zero if and only if $a$ and $b$ commute.

\begin{thm}\label{thm:6.2.3}
The following formula holds. 
$$
\left[ \alpha_{1}, \alpha_{2} \right] = 0. 
$$
\end{thm}
\begin{proof}
We have 
\begin{align*}
\alpha_{2} \alpha_{1} 
&= \left( \sum_{g \in \langle a \rangle} x_{g^{-1}} g - \sum_{g \in G \setminus \langle a \rangle} x_{g} g \right) \left( \sum_{g' \in G} x_{g'} g' \right) \\
&= \sum_{g' \in G} \left( \sum_{g \in \langle a \rangle} x_{g^{-1}} x_{g'} - \sum_{g \in G \setminus \langle a \rangle} x_{g} x_{g'} g g' \right) \\ 
&= \sum_{h \in G} \left( \sum_{g \in \langle a \rangle} x_{g^{-1}} x_{g^{-1} h} - \sum_{g \in G \setminus \langle a \rangle} x_{g} x_{g^{-1} h} \right) h. 
\end{align*}
By Lemmas~$\ref{lem:6.1.1}$ and~$\ref{lem:6.2.2}$, we have 
\begin{align*}
\alpha_{2} \alpha_{1} 
&= \sum_{h \in \langle a \rangle} \left( \sum_{g \in \langle a \rangle} x_{g} x_{h g} - \sum_{g \in \langle a \rangle} x_{g b} x_{h g b^{-1}} \right) \\ 
&= \alpha_{1} \alpha_{2}. 
\end{align*}
This completes the proof. 
\end{proof}

\begin{thm}\label{thm:6.2.4}
The following formula holds. 
$$
\left[ \alpha_{3}, \alpha_{4} \right] = 0. 
$$
\end{thm}
\begin{proof}
This theorem follows from Lemma $\ref{lem:6.1.2}$ and Theorem $\ref{thm:6.2.3}$. 
\end{proof}

\begin{lem}\label{lem:6.2.5}
The following formula holds. 
$$
\left[ \alpha_{1}, \alpha_{3} \right] 
= \sum_{h \in G} \sum_{g \in G} \chi_{3}(g^{-1} h) x_{g} x_{g^{-1} h} h 
- \sum_{h \in G} \sum_{g \in G} \chi_{3}(g) x_{g} x_{g^{-1} h} h. 
$$
\end{lem}
\begin{proof}
We compute 
\begin{align*}
\left[ \alpha_{1}, \alpha_{3} \right] 
&= \sum_{g \in G} x_{g} g \sum_{g' \in G} \chi_{3}(g') x_{g'} g' - \sum_{g \in G} \chi_{3}(g) x_{g} g \sum_{g' \in G} x_{g'} g' \\ 
&= \sum_{g \in G} \sum_{g' \in G} \chi_{3}(g') x_{g} x_{g'} g g' - \sum_{g \in G} \sum_{g' \in G} \chi_{3}(g) x_{g} x_{g'} g g' \\ 
&= \sum_{g \in G} \sum_{h \in G} \chi_{3}(g^{-1} h) x_{g} x_{g^{-1} h} h - \sum_{g \in G} \sum_{h \in G} \chi_{3}(g) x_{g} x_{g^{-1} h} h
\end{align*}
as required. 
\end{proof}

\begin{lem}\label{lem:6.2.6}
Let $h$ be $a^{k}$ with $k$ odd. 
Then, we have 
\begin{align*}
\sum_{g \in \langle a \rangle} \chi_{3}(g) x_{g} x_{g^{-1} h} = 0. 
\end{align*}
\end{lem}
\begin{proof}
We compute 
\begin{align*}
\sum_{g \in \langle a \rangle} \chi_{3}(g) x_{g} x_{g^{-1} h} 
&= \sum_{g \in \langle a \rangle} \chi_{3}(g^{-1} h) x_{g^{-1} h} x_{g} \\ 
&= - \sum_{g \in \langle a \rangle} \chi_{3}(g) x_{g} x_{g^{-1} h} 
\end{align*}
as required. 
\end{proof}

\begin{lem}\label{lem:6.2.7}
The coefficient of $h = a^{k}$ in $\left[ \alpha_{1}, \alpha_{3} \right]$ is 
$$
(-1)^{k} \sum_{g \in G \setminus \langle a \rangle} \chi_{3}(g^{-1}) x_{g} x_{g^{-1} h} 
- \sum_{g \in G \setminus \langle a \rangle} \chi_{3}(g) x_{g} x_{g^{-1} h}. 
$$
\end{lem}
\begin{proof}
We have 
\begin{align*}
& \sum_{g \in G} \chi_{3}(g^{-1} h) x_{g} x_{g^{-1} h} - \sum_{g \in G} \chi_{3}(g) x_{g} x_{g^{-1} h} \\
& \quad = (-1)^{k} \sum_{g \in G} \chi_{3}(g^{-1}) x_{g} x_{g^{-1} h} - \sum_{g \in G} \chi_{3}(g) x_{g} x_{g^{-1} h} \\
& \quad = (-1)^{k} \sum_{g \in \langle a \rangle} \chi_{3}(g) x_{g} x_{g^{-1} h} 
+ (-1)^{k} \sum_{g \in G \setminus \langle a \rangle} \chi_{3}(g^{-1}) x_{g} x_{g^{-1} h} \\ 
& \quad \quad - \sum_{g \in \langle a \rangle} \chi_{3}(g) x_{g} x_{g^{-1} h} 
- \sum_{g \in G \setminus \langle a \rangle} \chi_{3}(g) x_{g} x_{g^{-1} h}. 
\end{align*}
From Lemma $\ref{lem:6.2.6}$, this completes the proof. 
\end{proof}

By Lemma $\ref{lem:6.2.7}$, we know that 
$$
\left[ \alpha_{1}, \alpha_{3} \right] \neq 0. 
$$
From Lemma $\ref{lem:6.1.1}$, we have 
$$
\left[ \alpha_{2}, \alpha_{4} \right] \neq 0. 
$$
Namely, we have to be careful with the factors in Corollary $\ref{cor:5.2.4}$.

\begin{lem}\label{lem:6.2.8}
The coefficient of $h = a^{k} b$ in $\left[ \alpha_{1}, \alpha_{3} \right]$ is 
\begin{enumerate}
\item For the case that $m$ is odd. 
$$
\begin{cases}
(1 + i) \displaystyle\sum_{g \in \langle a \rangle} \chi_{3}(g) x_{g} x_{g h} 
- (1 + i) \displaystyle\sum_{g \in \langle a \rangle} \chi_{3}(g) x_{g} x_{g^{-1} h}, & {\rm if} \; k \; {\rm is \; odd}. \\
(1 - i) \displaystyle\sum_{g \in \langle a \rangle} \chi_{3}(g) x_{g} x_{g h} 
- (1 - i) \displaystyle\sum_{g \in \langle a \rangle} \chi_{3}(g) x_{g} x_{g^{-1} h}, & {\rm if} \; k \; {\rm is \; even}. 
\end{cases}
$$
\item For the case that $m$ is even. 
$$
\begin{cases}
-2 \displaystyle\sum_{g \in G} \chi_{3}(g) x_{g} x_{g^{-1} h}, & {\rm if} \; k \; {\rm is \; odd}. \\ 
0, & {\rm if} \; k \; {\rm is \; even}. 
\end{cases}
$$
\end{enumerate}
\end{lem}
\begin{proof}
\begin{enumerate}
\item For the case that $m$ is odd, 
we compute 
\begin{align*}
&\sum_{g \in G} \chi_{3}(g^{-1} h) x_{g} x_{g^{-1} h} - \sum_{g \in G} \chi_{3}(g) x_{g} x_{g^{-1} h} \\
& \quad = (-1)^{k} i \sum_{g \in G} \chi_{3}(g^{-1}) x_{g} x_{g^{-1} h} - \sum_{g \in G} \chi_{3}(g) x_{g} x_{g^{-1} h} \\ 
& \quad = (-1)^{k} i \sum_{g \in \langle a \rangle} \chi_{3}(g^{-1}) x_{g} x_{g^{-1} h} + (-1)^{k} i \sum_{g \in G \setminus \langle a \rangle} \chi_{3}(g^{-1}) x_{g} x_{g^{-1} h} \\ 
& \quad \quad - \sum_{g \in \langle a \rangle} \chi_{3}(g) x_{g} x_{g^{-1} h} - \sum_{g \in G \setminus \langle a \rangle} \chi_{3}(g) x_{g} x_{g^{-1} h} \\ 
& \quad = (-1)^{k} i \sum_{g \in \langle a \rangle} \chi_{3}(g) x_{g} x_{g^{-1} h} + (-1)^{k} i \sum_{g \in \langle a \rangle} \chi_{3}(b^{-1} g^{-1}) x_{g b} x_{b^{-1} g^{-1} h} \\ 
& \quad \quad - \sum_{g \in \langle a \rangle} \chi_{3}(g) x_{g} x_{g^{-1} h} - \sum_{g \in \langle a \rangle} \chi_{3}(g b) x_{g b} x_{b^{-1} g^{-1} h} \\ 
& \quad = (-1)^{k} i \sum_{g \in \langle a \rangle} \chi_{3}(g) x_{g} x_{g^{-1} h} + (-1)^{k} \sum_{g \in \langle a \rangle} \chi_{3}(g) x_{g b} x_{a^{-k} g} \\ 
& \quad \quad - \sum_{g \in \langle a \rangle} \chi_{3}(g) x_{g} x_{g^{-1} h} - i \sum_{g \in \langle a \rangle} \chi_{3}(g) x_{g b} x_{a^{-k} g} \\ 
& \quad = (-1)^{k} i \sum_{g \in \langle a \rangle} \chi_{3}(g) x_{g} x_{g^{-1} h} + (-1)^{k} \sum_{g \in \langle a \rangle} \chi_{3}(g a^{k}) x_{g a^{k} g} x_{g} \\ 
& \quad \quad - \sum_{g \in \langle a \rangle} \chi_{3}(g) x_{g} x_{g^{-1} h} - i \sum_{g \in \langle a \rangle} \chi_{3}(g a^{k}) x_{g a^{k} b} x_{g} \\ 
& \quad = (-1)^{k} i \sum_{g \in \langle a \rangle} \chi_{3}(g) x_{g} x_{g^{-1} h} + \sum_{g \in \langle a \rangle} \chi_{3}(g) x_{g h} x_{g} \\ 
& \quad \quad - \sum_{g \in \langle a \rangle} \chi_{3}(g) x_{g} x_{g^{-1} h} + (-1)^{k + 1} i \sum_{g \in \langle a \rangle} \chi_{3}(g) x_{g h} x_{g} 
\end{align*}
as required. 
\item For the case that $m$ is even, 
we compute 
\begin{align*}
&\sum_{g \in G} \chi_{3}(g^{-1} h) x_{g} x_{g^{-1} h} - \sum_{g \in G} \chi_{3}(g) x_{g} x_{g^{-1} h} \\ 
& \quad = (-1)^{k} \sum_{g \in G} \chi_{3}(g) x_{g} x_{g^{-1} h} - \sum_{g \in G} \chi_{3}(g) x_{g} x_{g^{-1} h} 
\end{align*} 
as required. 
\end{enumerate}
\end{proof}

\begin{lem}\label{lem:6.2.9}
The following formula holds. 
\begin{align*}
\left[ \alpha_{1}, \alpha_{4} \right] 
&= \sum_{h \in G} \sum_{g \in \langle a \rangle} \chi_{3}(g) x_{g} x_{h g} h 
- \sum_{h \in G} \sum_{g \in G \setminus \langle a \rangle} \chi_{3}(g) x_{g} x_{h g^{-1}} h \\ 
& \quad - \sum_{h \in G} \sum_{g \in \langle a \rangle} \chi_{3}(g) x_{g} x_{g h} h 
+ \sum_{h \in G} \sum_{g \in G \setminus \langle a \rangle} \chi_{3}(g) x_{g} x_{g^{-1} h} h. 
\end{align*}
\end{lem}
\begin{proof}
We compute 
\begin{align*}
\left[ \alpha_{1}, \alpha_{4} \right] 
&= \sum_{g' \in G} x_{g'} g' 
\left( \sum_{g \in \langle a \rangle} \chi_{4}(g) x_{g^{-1}} g + \sum_{g \in G \setminus \langle a \rangle} \chi_{4}(g) x_{g} g \right) \\
& \quad - \left( \sum_{g \in \langle a \rangle} \chi_{4}(g) x_{g^{-1}} g + \sum_{g \in G \setminus \langle a \rangle} \chi_{4}(g) x_{g} g \right) 
\sum_{g' \in G} x_{g'} g' \\ 
&= \sum_{g' \in G} \sum_{g \in \langle a \rangle} \chi_{4}(g) x_{g'} x_{g^{-1}} g' g 
+ \sum_{g' \in G} \sum_{g \in G \setminus \langle a \rangle} \chi_{4}(g) x_{g'} x_{g} g' g \\ 
& \quad - \sum_{g \in \langle a \rangle} \sum_{g' \in G} \chi_{4}(g) x_{g^{-1}} x_{g'} g g' 
- \sum_{g \in G \setminus \langle a \rangle} \sum_{g' \in G} \chi_{4}(g) x_{g} x_{g'} g g' \\ 
&= \sum_{h \in G} \sum_{g \in \langle a \rangle} \chi_{4}(g) x_{h g^{-1}} x_{g^{-1}} h 
+ \sum_{h \in G} \sum_{g \in G \setminus \langle a \rangle} \chi_{4}(g) x_{h g^{-1}} x_{g} h \\ 
& \quad - \sum_{g \in \langle a \rangle} \sum_{h \in G} \chi_{4}(g) x_{g^{-1}} x_{g^{-1} h} h 
- \sum_{g \in G \setminus \langle a \rangle} \sum_{h \in G} \chi_{4}(g) x_{g} x_{g^{-1} h} h \\ 
&= \sum_{h \in G} \sum_{g \in \langle a \rangle} \chi_{3}(g) x_{h g^{-1}} x_{g^{-1}} h 
- \sum_{h \in G} \sum_{g \in G \setminus \langle a \rangle} \chi_{3}(g) x_{h g^{-1}} x_{g} h \\ 
& \quad - \sum_{g \in \langle a \rangle} \sum_{h \in G} \chi_{3}(g) x_{g^{-1}} x_{g^{-1} h} h 
+ \sum_{g \in G \setminus \langle a \rangle} \sum_{h \in G} \chi_{3}(g) x_{g} x_{g^{-1} h} h \\ 
&= \sum_{h \in G} \sum_{g \in \langle a \rangle} \chi_{3}(g) x_{g} x_{h g} h 
- \sum_{h \in G} \sum_{g \in G \setminus \langle a \rangle} \chi_{3}(g) x_{g} x_{h g^{-1}} h \\ 
& \quad - \sum_{h \in G} \sum_{g \in \langle a \rangle} \chi_{3}(g) x_{g} x_{g h} h 
+ \sum_{h \in G} \sum_{g \in G \setminus \langle a \rangle} \chi_{3}(g) x_{g} x_{g^{-1} h} h
\end{align*}
as required. 
\end{proof}

\begin{lem}\label{lem:6.2.10}
The coefficient of $h = a^{k}$ in $\left[ \alpha_{1}, \alpha_{4} \right]$ is 
$$
(-1)^{k + 1} \sum_{g \in G \setminus \langle a \rangle} \chi_{3}(g^{-1}) x_{g} x_{g^{-1} h} + \sum_{g \in G \setminus \langle a \rangle} \chi_{3}(g) x_{g} x_{g^{-1} h}. 
$$
\end{lem}
\begin{proof}
We compute 
\begin{align*}
&\sum_{g \in \langle a \rangle} \chi_{3}(g) x_{g} x_{h g} - \sum_{g \in G \setminus \langle a \rangle} \chi_{3}(g) x_{g} x_{h g^{-1}} \\ 
\quad \quad & - \sum_{g \in \langle a \rangle} \chi_{3}(g) x_{g} x_{g h} + \sum_{g \in G \setminus \langle a \rangle} \chi_{3}(g) x_{g} x_{g^{-1} h} \\ 
= & - \sum_{g \in G \setminus \langle a \rangle} \chi_{3}(g) x_{g} x_{h g^{-1}} + \sum_{g \in G \setminus \langle a \rangle} \chi_{3}(g) x_{g} x_{g^{-1} h} \\ 
= & - \sum_{g \in G \setminus \langle a \rangle} \chi_{3}(g^{-1} h) x_{g^{-1} h} x_{h h^{-1} g} 
+ \sum_{g \in G \setminus \langle a \rangle} \chi_{3}(g) x_{g} x_{g^{-1} h} \\ 
= & \; (-1)^{k + 1} \sum_{g \in G \setminus \langle a \rangle} \chi_{3}(g^{-1}) x_{g} x_{g^{-1} h} 
+ \sum_{g \in G \setminus \langle a \rangle} \chi_{3}(g) x_{g} x_{g^{-1} h}
\end{align*}
as required. 
\end{proof}

From Lemma $\ref{lem:6.2.10}$, we know that 
$$
\left[ \alpha_{1}, \alpha_{4} \right] \neq 0. 
$$
From Theorems~$\ref{thm:6.2.3}$ and~$\ref{thm:6.2.4}$, 
we think that degree one representations form natural pairing $(\chi_{1}, \chi_{2})$ and $(\chi_{3}, \chi_{4})$. 
Furthermore, we will have another interesting algebraic properties soon. 

\begin{lem}\label{lem:6.2.11}
The coefficient of $h = a^{k} b$ in $\left[ \alpha_{1}, \alpha_{4} \right]$ is 
\begin{enumerate}
\item For the case that $m$ is odd. 
$$
\begin{cases} 
- (1 + i) \displaystyle\sum_{g \in \langle a \rangle} \chi_{3}(g) x_{g} x_{g h} + (1 + i) \displaystyle\sum_{g \in \langle a \rangle} \chi_{3}(g) x_{g} x_{g^{-1} h}, 
 & {\rm if} \; k \; {\rm is \; odd}. \\
- (1 - i) \displaystyle\sum_{g \in \langle a \rangle} \chi_{3}(g) x_{g} x_{g h} + (1 - i) \displaystyle\sum_{g \in \langle a \rangle} \chi_{3}(g) x_{g} x_{g^{-1} h}, 
& {\rm if} \; k \; {\rm is \; even}. 
\end{cases} 
$$
\item For the case that $m$ is even. 
$$
\begin{cases}
2 \displaystyle\sum_{g \in G} \chi_{3}(g) x_{g} x_{g^{-1} h}, & {\rm if} \; $k$ \; \rm{is \; odd}. \\ 
0, & {\rm if} \; $k$ \; {\rm is \; even}. 
\end{cases}
$$
\end{enumerate}
\end{lem}

\begin{proof}
\allowdisplaybreaks 
\begin{enumerate}
\item For the case that $m$ is odd, 
we compute 
\begin{align*}
&\sum_{g \in \langle a \rangle} \chi_{3}(g) x_{g} x_{h g} - \sum_{g \in G \setminus \langle a \rangle} \chi_{3}(g) x_{g} x_{h g^{-1}} \\ 
& \quad \quad - \sum_{g \in \langle a \rangle} \chi_{3}(g) x_{g} x_{g h} + \sum_{g \in G \setminus \langle a \rangle} \chi_{3}(g) x_{g} x_{g^{-1} h} \\
& \quad = \sum_{g \in \langle a \rangle} \chi_{3}(g) x_{g} x_{g^{-1} h} - \sum_{g \in \langle a \rangle} \chi_{3}(g b) x_{g b} x_{h b^{-1} g^{-1}} \\ 
& \quad \quad - \sum_{g \in \langle a \rangle} \chi_{3}(g) x_{g} x_{g h} + \sum_{g \in \langle a \rangle} \chi_{3}(g b) x_{g b} x_{b^{-1} g^{-1} h} \\ 
& \quad = \sum_{g \in \langle a \rangle} \chi_{3}(g) x_{g} x_{g^{-1} h} - i \sum_{g \in \langle a \rangle} \chi_{3}(g) x_{g b} x_{a^{k} g^{-1}} \\
& \quad \quad - \sum_{g \in \langle a \rangle} \chi_{3}(g) x_{g} x_{g h} + i \sum_{g \in \langle a \rangle} \chi_{3}(g) x_{g b} x_{a^{-k} g} \\ 
& \quad = \sum_{g \in \langle a \rangle} \chi_{3}(g) x_{g} x_{g^{-1} h} - i \sum_{g \in \langle a \rangle} \chi_{3}(g a^{k}) x_{g a^{k} b} x_{g^{-1}} \\ 
& \quad \quad - \sum_{g \in \langle a \rangle} \chi_{3}(g) x_{g} x_{g h} + (-1)^{k} i \sum_{g \in \langle a \rangle} \chi_{3}(g) x_{g h} x_{g} \\ 
& \quad = \sum_{g \in \langle a \rangle} \chi_{3}(g) x_{g} x_{g^{-1} h} - (-1)^{k} i \sum_{g \in \langle a \rangle} \chi_{3}(g) x_{g h} x_{g^{-1}} \\ 
& \quad \quad - \sum_{g \in \langle a \rangle} \chi_{3}(g) x_{g} x_{g h} + (-1)^{k} i \sum_{g \in \langle a \rangle} \chi_{3}(g) x_{g} x_{g h} 
\end{align*}
as required. 
\item For the case that $m$ is even, 
we compute 
\begin{align*} 
&\sum_{g \in \langle a \rangle} \chi_{3}(g) x_{g} x_{h g} - \sum_{g \in G \setminus \langle a \rangle} \chi_{3}(g) x_{g} x_{h g^{-1}} \\ 
& \quad \quad - \sum_{g \in \langle a \rangle} \chi_{3}(g) x_{g} x_{g h} + \sum_{g \in G \setminus \langle a \rangle} \chi_{3}(g) x_{g} x_{g^{-1} h} \\ 
& \quad = \sum_{g \in \langle a \rangle} \chi_{3}(g) x_{g} x_{g^{-1} h} - \sum_{g \in \langle a \rangle} \chi_{3}(g b) x_{g b} x_{h b^{-1} g^{-1}} \\ 
& \quad \quad - \sum_{g \in \langle a \rangle} \chi_{3}(g b) x_{b^{-1} b g} x_{g h} + \sum_{g \in G \setminus \langle a \rangle} \chi_{3}(g) x_{g} x_{g^{-1} h} \\ 
& \quad = \sum_{g \in G} \chi_{3}(g) x_{g} x_{g^{-1} h} - \sum_{g \in \langle a \rangle} \chi_{3}(g) x_{g b} x_{a^{k} g^{-1}} \\ 
& \quad \quad - \sum_{g \in \langle a \rangle} \chi_{3}(g b) x_{b^{-1} g^{-1} b} x_{g a^{k} b} \\ 
& \quad = \sum_{g \in G} \chi_{3}(g) x_{g} x_{g^{-1} h} - \sum_{g \in \langle a \rangle} \chi_{3}(g a^{k}) x_{g a^{k} b} x_{g^{-1}} \\ 
& \quad \quad - \sum_{g \in \langle a \rangle} \chi_{3}(g a^{-k} b) x_{b^{-1} g^{-1} a^{k} b} x_{g b} \\ 
& \quad = \sum_{g \in G} \chi_{3}(g) x_{g} x_{g^{-1} h} + (-1)^{k + 1} \sum_{g \in \langle a \rangle} \chi_{3}(g) x_{g} x_{g^{-1} h} \\ 
& \quad \quad + (-1)^{k + 1} \sum_{g \in G \setminus \langle a \rangle} \chi_{3}(g) x_{g^{-1} h} x_{g} \\ 
& \quad = \sum_{g \in G} \chi_{3}(g) x_{g} x_{g^{-1} h} + (-1)^{k + 1} \sum_{g \in G} \chi_{3}(g) x_{g} x_{g^{-1} h} 
\end{align*}
as required. 
\end{enumerate}
\end{proof}

We have the following algebraic properties. 

\begin{thm}\label{thm:6.2.13}
The formulas 
\begin{align*}
\left[\alpha_{1}, \alpha_{3} + \alpha_{4} \right] &= 0, \\
\left[\alpha_{2}, \alpha_{3} + \alpha_{4} \right] &= 0, \\ 
\left[\alpha_{3}, \alpha_{1} + \alpha_{2} \right] &= 0, \\ 
\left[\alpha_{4}, \alpha_{1} + \alpha_{2} \right] &= 0 
\end{align*}
hold. 
\end{thm}
\begin{proof}
By Lemmas~$\ref{lem:6.2.7},\ \ref{lem:6.2.8},\ \ref{lem:6.2.10}$ and $\ref{lem:6.2.11}$, 
we have 
$$
\left[ \alpha_{1}, \alpha_{3} + \alpha_{4} \right] = 0. 
$$
From Lemmas~$\ref{lem:6.1.1}$ and~$\ref{lem:6.1.2}$, 
this completes the proof. 
\end{proof}

\clearpage

\thanks{\bf{Acknowledgments}}
I am deeply grateful to Prof. Hiroyuki Ochiai and Prof. Minoru Itoh who provided the helpful comments and suggestions. 
Also, I would like to thank my college in the Graduate School of Mathematics of Kyushu University, 
in particular Cid Reyes, Tomoyuki Tamura, Yuka Suzuki for comments and suggestions. 
I would also like to express my gratitude to my family for their moral support and warm encouragements.

\medskip
\begin{flushleft}
Naoya Yamaguchi\\
Graduate School of Mathematics\\
Kyushu University\\
Nishi-ku, Fukuoka 819-0395 \\
Japan\\
n-yamaguchi@math.kyushu-u.ac.jp
\end{flushleft}


\begin{thebibliography}{99}

\bibitem{benjamin} 
Benjamin Steinberg, {\it Representation Theory of Finite Groups}. Springer, 2012. 

\bibitem{Augmentation quotients for complex representation rings of dihedral groups} 
C. Shan, C. Hong, T. Guoping, 
Augmentation quotients for complex representation rings of dihedral groups, Frontiers of Mathematics in China. 
Vol. 7, pp. 1-18, 2012. 

\bibitem{on the group determinant} 
K. W. Jonson, On the group determinant, Mathematical Proceedings of the Cambridge Philosophical Society. 
Vol. 109, pp. 299-311, 1991. 

\bibitem{On the Structure of Augmentation Quotient Groups for the Generalized Quaternion Group} 
Z. Qingxia, Y. Hong, 
On the Structure of Augmentation Quotient Groups for the Generalized Quaternion Group, 
Algebra Colloquium. 
Vol. 19, pp. 137-148, 2012. 


\end{thebibliography}
\end{document}